\documentclass[11pt]{amsart}
\usepackage{geometry}
\usepackage{amscd,amssymb,verbatim,xcolor}
\usepackage{longtable, multirow}
\usepackage{soul,cancel}

\date{\today}

\newcommand\bC{{\mathbb C}}

\newcommand\fa{{\mathfrak a}}

\newcommand\fg{{\mathfrak g}}
\newcommand\fh{{\mathfrak h}}
\newcommand\frk{{\mathfrak k}}

\newcommand\fp{{\mathfrak p}}

\newcommand\ft{{\mathfrak t}}

\newtheorem{theorem}{Theorem}[section]
\newtheorem{algorithm}[theorem]{Algorithm}

\newtheorem{corollary}[theorem]{Corollary}

\newtheorem{definition}[theorem]{Definition}
\newtheorem{example}[theorem]{Example}
\newtheorem{lemma}[theorem]{Lemma}
\newtheorem{proposition}[theorem]{Proposition}

\begin{document}
\title[Scattered Representations]{Scattered representations of $SL(n, \bC)$}
\author{Chao-Ping Dong}
\author{Kayue Daniel Wong}

\address[Dong]{School of Mathematical Sciences, Soochow University, Suzhou 215006,
P.~R.~China}
\email{chaopindong@163.com}

\address[Wong]{School of Science and Engineering, The Chinese University of Hong Kong, Shenzhen,
Guangdong 518172, P. R. China}
\email{kayue.wong@gmail.com}

\begin{abstract}
Let $G$ be $SL(n, \bC)$. The unitary dual $\widehat{G}$ was classified by Vogan in the 1980s. This paper aims to describe the Zhelobenko parameters and the spin-lowest $K$-types of the scattered representations of $G$,
which lie at the heart of $\widehat{G}^d$---the set of all the equivalence classes of irreducible unitary representations of $G$ with non-vanishing Dirac cohomology.
As a consequence, we will verify a couple of conjectures of the first-named author for $G$.
\end{abstract}

\maketitle
\setcounter{tocdepth}{1}

\section{Introduction}\label{sec:intro}
\subsection{Preliminaries on complex simple Lie groups}
Let $G$ be a complex connected simple Lie group, and $H$ be a Cartan subgroup of $G$. Let $\fg_0$ and $\fh_0$ be the Lie algebra of $G$ and $H$ respectively, and we drop the subscripts to stand for the complexified Lie algebras. We adopt a positive root system $\Delta^+(\fg_0, \fh_0)$, and let $\varpi_1, \dots, \varpi_{\mathrm{rank}(\fg_0)}$ be the corresponding fundamental weights with $\rho=\varpi_1+\cdots+\varpi_{\mathrm{rank}(\fg_0)}$ being the half sum of positive roots.

Fix a Cartan involution $\theta$ on $G$ such that its fixed points form a maximal compact subgroup $K$ of $G$. Then on the Lie algebra level, we have the Cartan decomposition
$$
\fg_0=\frk_0+\fp_0.
$$

We denote by $\langle\cdot, \cdot\rangle$ the Killing form on $\fg_0$. This form is negative definite on $\frk_0$ and positive definite on $\fp_0$. Moreover, $\frk_0$ and $\fp_0$ are orthogonal to each other under $\langle\cdot, \cdot\rangle$. We shall denote by $\|\cdot\|$ the norm corresponding to the Killing form.

Let $H=TA$ be the Cartan decomposition of $H$, with $\fh_0=\ft_0+\fa_0$. We make the following identifications:
\begin{equation}\label{identifction}
\fh\cong \fh_0\times \fh_0, \quad \ft=\{(x, -x): x\in\fh_0\}, \quad \fa\cong\{(x, x): x\in \fh_0\}.
\end{equation}
Take an arbitrary pair $(\lambda_L, \lambda_R)\in \fh_0^*\times \fh_0^*$ such that $\mu:=\lambda_L-\lambda_R$ is integral. Denote by $\{\mu\}$ the unique dominant weight to which $\mu$ is conjugate under the action of the Weyl group $W$. Write $\nu:=\lambda_L + \lambda_R$. We can view $\mu$ as a weight of $T$ and $\nu$ a character of $A$. Put
$$
I(\lambda_L, \lambda_R):={\rm Ind}_B^G(\bC_{\mu}\otimes \bC_{\nu}\otimes {\bf 1})_{K-{\rm finite}},
$$
where $B$ is the Borel subgroup of $G$ determined by $\Delta^+(\fg_0, \fh_0)$. It is not hard to show that $V_{\{\mu\}}$, the $K$-type with highest weight $\{\mu\}$, occurs exactly once in $I(\lambda_L, \lambda_R)$. Let $J(\lambda_L, \lambda_R)$ be the unique irreducible subquotient of $I(\lambda_L, \lambda_R)$ containing $V_{\{\mu\}}$. By \cite{Zh}, every irreducible admissible $(\fg, K)$-module has the form $J(\lambda_L, \lambda_R)$. Indeed, up to equivalence, $J(\lambda_L, \lambda_R)$ is the unique irreducible admissible $(\fg, K)$-module with infinitesimal character the $W \times W$ orbit of $(\lambda_L, \lambda_R)$, and lowest $K$-type $V_{\{\lambda_L-\lambda_R\}}$. We will refer to the pair $(\lambda_L, \lambda_R)$ as the {\it Zhelobenko parameter} for the module $J(\lambda_L, \lambda_R)$.

\subsection{Dirac cohomology}
Fix an orthonormal basis $Z_1, \dots, Z_l$ of $\fp_0$ with respect to the inner product on $\fp_0$ induced by $\langle\cdot, \cdot\rangle$. Let $U(\fg)$ be the universal enveloping algebra of $\fg$, and put $C(\fp)$ as the Clifford algebra of $\fp$. One checks that
\begin{equation}\label{Dirac-operator}
D:=\sum_{i=1}^{l} Z_i\otimes Z_i \in U(\fg)\otimes C(\fp)
\end{equation}
is independent of the choice of the orthonormal basis $Z_1, \dots, Z_l$. The operator $D$, called the \textit{Dirac operator}, was introduced by Parthasarathy \cite{P1}. By construction, $D^2$ is a natural Laplacian on $G$, which gives rise to the Parthasarathy's Dirac inequality (see \eqref{Dirac-inequality} below). The inequality is very effective for detecting non-unitarity of $(\fg,K)$-modules, but is by no means sufficient to classify all (non-)unitary modules.

To sharpen the Dirac inequality, and to offer a better understanding of the unitary dual, Vogan formulated the notion of Dirac cohomology in 1997 \cite{V2}. Let ${\rm Ad}: K\rightarrow SO(\fp_0)$ be the adjoint map, ${\rm Spin}\ \fp_0$ be the spin group of $\fp_0$, and denote by $p: {\rm Spin}\ \fp_0\rightarrow SO(\fp_0)$ the spin double covering map. Put
$$
\widetilde{K}:=\{(k,s)\in K\times {\rm Spin} \, \fp_0\mid {\rm Ad}(k)=p(s)\}.
$$
As in the case of $K$-types, we will refer to an irreducible $\widetilde{K}$-type with highest weight $\delta$ as $V_{\delta}$.

Let $\pi$ be any admissible $(\fg, K)$-module, and $S$ be the spin module of $C(\fp)$.  Then $U(\fg)\otimes C(\fp)$, in particular the Dirac operator $D$, acts on $\pi\otimes S$. Now the \textit{Dirac cohomology} is defined as the $\widetilde{K}$-module
\begin{equation}\label{Dirac-cohomology}
H_D(\pi):={\rm Ker} D/({\rm Ker} D \cap {\rm Im}  D).
\end{equation}
It is evident from the definition that Dirac cohomology is an invariant for admissible $(\fg, K)$-modules. To compute this invariant, the Vogan conjecture, proved by Huang and Pand\v{z}i\'{c} \cite{HP1}, says that whenever $H_D(\pi) \neq 0$, one would have
\begin{equation}\label{thm-HP}
\gamma+\rho=w\Lambda,
\end{equation}
where $\Lambda$ is the infinitesimal character of $\pi$, $\gamma$ is the highest weight of any $\widetilde{K}$-type in $H_D(\pi)$, and $w$ is some element of $W$.

It turns out that many interesting $(\fg,K)$-modules $\pi$, such as some $A_q(\lambda)$-modules and all the highest weight modules, have non-zero Dirac cohomology (see \cite{HKP}, \cite{HPP}). One would therefore like to classify all representations with non-zero Dirac cohomology.

\subsection{Spin-lowest $K$-type}
From now on, we set $\pi$ as an irreducible unitary $(\fg, K)$-module with infinitesimal character $\Lambda$.
In order to get a clearer picture on $H_D(\pi)$, the first-named author introduced the notion of spin-lowest $K$-types.
Given an arbitrary $K$-type $V_{\delta}$, its spin norm is defined as
\begin{equation}\label{spin-norm}
\|\delta\|_{\rm spin}:=\|\{\delta-\rho\}+\rho\|.
\end{equation}
Then a $K$-type $V_{\tau}$ occurring in $\pi$ is called a \textit{spin-lowest $K$-type} of $\pi$ if it achieves the minimum spin norm among all the $K$-types showing up in $\pi$.

As an application of spin-lowest $K$-type, note that $D$ is self-adjoint on the unitarizable module $\pi\otimes S$. By writing out $D^2$ carefully, and by using the \emph{PRV-component} \cite{PRV}, we can rephrase \textit{Parthasarathy's Dirac operator inequality} \cite{P2} as follows:
\begin{equation}\label{Dirac-inequality}
\|\delta\|_{\rm spin}\geq \|\Lambda\|,
\end{equation}
where $V_{\delta}$ is any $K$-type. Moreover, one can deduce from \cite[Theorem 3.5.3]{HP2} that $H_D(\pi)\neq 0$ if and only if the spin-lowest $K$-types $V_{\tau}$ attain the lower bound of Equation \eqref{Dirac-inequality}. In such cases, $V_{\{\tau - \rho\}}$ will show up in $H_D(\pi)$. Put in a different way, the spin-lowest $K$-types of $\pi$ are exactly the $K$-types contributing to $H_D(\pi)$  whenever the cohomology is non-vanishing (see Proposition 2.3 of \cite{D1} for more details).

\subsection{Scattered representations} \label{scattered}
Based on the studies \cite{BP,DD}, we are interested in the  irreducible unitarizable  $(\fg, K)$-modules $J(\lambda, -s\lambda)$ such that
\begin{itemize}
\item[(i)] the weight 2$\lambda$ is dominant integral, i.e., $2\lambda=\sum_{i=1}^{\mathrm{rank}(\fg_0)}c_i\varpi_i$, where each $c_i$ is a positive integer;

\item[(ii)] the element $s\in W$ is an involution such that each simple reflection $s_i$, $1\leq i\leq \mathrm{rank}(\fg_0)$, occurs in one (thus in each) reduced expression of $s$;

\item[(iii)] the module has non-zero Dirac cohomology, i.e., $H_D(J(\lambda, -s\lambda))\neq 0$, or equivalently, there exists
a $K$-type $V_{\tau}$ in $J(\lambda, -s\lambda)$ such that
\begin{equation} \label{spin=2lambda}
\|\tau\|_{\rm spin} = \|(\lambda, -s\lambda)\| = \|2\lambda\|
\end{equation}
\end{itemize}

According to \cite{DD}, there are only finitely many such representations, which are called the \textit{scattered representations}.

These representations lie at the heart of $\widehat{G}^d$ --- the set of all the irreducible unitary $(\fg, K)$-modules of $G$ with non-zero Dirac cohomology up to equivalence. Namely, by Theorem A of \cite{DD}, any member of $\widehat{G}^d$ is either a scattered representation, or it is cohomologically induced from a scattered representation tensored with a suitable unitary character of the Levi factor of a certain proper $\theta$-stable parabolic subgroup. In the latter case, one can easily trace the spin-lowest $K$-types along with the Dirac cohomology of the modules before and after induction.
It is therefore of interest to have a good understanding of scattered representations.

\subsection{Overview}
In this manuscript, we focus on Lie groups $G$ of Type $A$. For convenience, we will start from the group $GL(n, \bC)$, written as $GL(n)$ for short. In this case, Vogan classified the unitary dual. The part that we need can be described as follows.

\begin{theorem}[\cite{V1}] \label{unitary}
All irreducible unitary representations of $GL(n)$ with regular half-integral infinitesimal
characters are parabolically induced from a unitary character, i.e. they are of the form
$${\rm Ind}_{(\prod_{i=0}^m GL(a_i))U}^{GL(n)}(\bigotimes_{i=0}^m {\det}^{p_i} \otimes \mathbf{1})$$
for some $a_i \in \mathbb{Z}_{>0}$ and $p_i \in \mathbb{Z}$. For simplicity, we will write the parabolically
induced module ${\rm Ind}_{LU}^{G}(\pi \otimes \mathbf{1})$ as ${\rm Ind}_{L}^{G}(\pi)$ for the rest of the manuscript.
\end{theorem}
Using  \cite[Theorem 2.4]{BP}, all such $\pi$ have non-zero Dirac cohomology. Moreover, \cite{BDW} proved Conjecture 4.1 of \cite{BP}, which says
$$H_D(\pi) = 2^{[\frac{\mathrm{rank}(\fg_0)}{2}]}V_{\{\tau - \rho\}},$$
where $V_{\tau}$ is the {\it unique} spin-lowest $K$-type appearing in $\pi$ {\it with multiplicity one}. However, it
is not clear what $V_{\tau}$ is like from the calculations in \cite{BDW}.

In Section 2, we will give an algorithm to compute $V_{\tau}$ for all such $\pi$ (see Proposition \ref{prop-spin-lowest}).  In Section 3, we will see how the calculations for $GL(n)$ in Section 2 can be translated to $SL(n)$, which gives a combinatorial description of scattered representations of $SL(n)$ (Proposition \ref{prop:scattered}). As a result, we prove the following:
\begin{itemize}
\item The spin-lowest $K$-type of each scattered representation of $SL(n)$ is {\it unitarily small} in the sense of Salamanca-Riba and Vogan \cite{SV} (Corollary \ref{cor-u-small}); and
\item The number of scattered representations of $SL(n)$ is equal to $2^{n-2}$ (Corollary \ref{cor-number}).
\end{itemize}
This verifies Conjecture C of \cite{DD} in the case of $SL(n)$, and proves Conjecture 5.2 of \cite{D2} respectively.

It is worth noting that for any non-trivial scattered representation,  its spin-lowest $K$-type lives deeper than, and differs from the lowest $K$-type. We hope the effort here will shed some light on the real case in future.


\section{An algorithm computing the spin-lowest $K$-types}
In this section, we give an algorithm to find the spin-lowest $K$-types of
the irreducible unitary modules of $GL(n)$ given by Theorem \ref{unitary}. We use a {\bf chain}
$$\mathcal{C} := \{c, c-2 \dots, c-(2k-2), c-2k\},$$
where $c, k \in \mathbb{Z}$ with $k >0$, to denote the Zhelobenko parameter
$$\begin{pmatrix} \lambda \\ -w_0\lambda \end{pmatrix} = \begin{pmatrix}
\frac{c}{2} & \frac{c}{2} - 1& \dots & \frac{c}{2} - (k-1) & \frac{c}{2} - k \\
-\frac{c}{2} + k & -\frac{c}{2} + (k-1) & \dots & -\frac{c}{2} + 1 & -\frac{c}{2} \\
\end{pmatrix}.$$

Note that the entries of $\mathcal{C}$ are precisely equal to $2\lambda$. Also,
this parameter corresponds to the one-dimensional module ${\det}^{c-k}$ of $GL(k+1)$.
Consequently, Theorem \ref{unitary} implies that the Zhelobenko parameters of all irreducible unitary modules
with regular half-integral infinitesimal character can be expressed by the chains
$$(\lambda,-s\lambda) = \bigcup_{i=0}^m \mathcal{C}_i,$$
where all the entries of $\mathcal{C}_i$ are disjoint.

In order to understand the spin-lowest $K$-types of these modules of $GL(n)$, we make
the following:
\begin{definition}
\begin{itemize}
\item[(a)] Two chains $\mathcal{C}_1 = \{A, \dots, a\}$, $\mathcal{C}_2 = \{B, \dots, b\}$ are {\bf linked} if
the entries of $\mathcal{C}_1$ and $\mathcal{C}_2$ are disjoint satisfying
$$A > B > a \ \ \ \ \text{or} \ \ \ \ B > A > b.$$
\item[(b)] We say a union of chains  $\displaystyle \bigcup_{i \in I} \mathcal{C}_i$ is {\bf interlaced} if for all $i \neq j$ in $I$,
there exist indices $i = i_0, i_1, \dots, i_m = j$ in $I$ such that $\mathcal{C}_{i_{l-1}}$ and
$\mathcal{C}_{i_{l}}$ are linked for all $1 \leq l \leq m$.
(By convention, we also let the single chain $\mathcal{C}_1$ be interlaced).
\end{itemize}
\end{definition}

For example, the parameter $\{9,7,5\} \cup \{6,4,2\} \cup \{3,1\}$
is interlaced, while the parameter $\{10,8\} \cup \{9,7\} \cup \{6,4\} \cup \{5,3,1\}$
is not interlaced.

We are now in the position to describe the spin-lowest $K$-types of
the unitary modules in Theorem \ref{unitary} using chains.
\begin{algorithm} \label{alg:spinlkt}
Let $J(\lambda,-s\lambda)$ be an irreducible unitary module of $GL(n)$ in Theorem \ref{unitary} with $(\lambda,-s\lambda) = \bigcup_{i=0}^m \mathcal{C}_i$, where
$$\mathcal{C}_i := \{k_i + (d_i -1), \dots, k_i - (d_i - 1)\} = \{C_{i,1},\dots,C_{i,d_i}\}$$
is a chain with average value $k_i$ and length $d_i$. Then the lowest $K$-type is equal to (a $W$-conjugate of) $(\mathcal{T}_0, \dots, \mathcal{T}_m)$, where
$$\mathcal{T}_i := (\underbrace{k_i, \dots, k_i}_{d_i}).$$
By re-indexing the chains when necessary, we may and we will assume that
\begin{equation} \label{eq:order}
\mbox{ for any }\ 0\leq i<j\leq m, \quad k_i > k_j \ \ \text{or}\ \  d_i < d_j\ \text{if}\ \ k_i = k_j.
\end{equation}

Let us
change the coordinates of $\mathcal{T}_i$ and $\mathcal{T}_j$ for all pairs of linked chains
$\mathcal{C}_i$ and $\mathcal{C}_j$ such that $i<j$ by the following rule:
\begin{itemize}
\item[(a)] If $C_{i,1} > C_{j,1} \geq C_{j,d_j} > C_{i,d_i}$, i.e.
\begin{align*}
\{C_{i,1},\ \ \dots,\ \ C_{i,d_i-p}&,\ \  \overbrace{C_{i,d_i-p+1},\ \dots \dots,\ \ C_{i,d_i}}^{p} \} \\
&\{C_{j,1}, \ \ \dots,\ \  C_{j,d_j}\}
\end{align*}
with $C_{j,1} = C_{i,d_i} + 2p-1$ and $d_j \leq p$, then we change the coordinates of $\mathcal{T}_i$
and $\mathcal{T}_j$ into:
\[\boxed{\begin{aligned}
\mathcal{T}_i' &: (*,\ \ \dots,\ \ *,\ \overbrace{k_i+p,\ k_i+(p-1),\ \dots,\ k_i+(p - d_j + 1),\ * ,\ \dots,\ *}^{p} ) \\
\mathcal{T}_j' &: \ \ \ \ \ \ \ \ \ \ \ \ (k_j - p,\ k_j - (p-1),\ \dots,\ k_j-(p - d_j+1)),
\end{aligned}}\]
where the entries marked by $*$ remain unchanged.

\smallskip

\item[(b)] If $C_{i,1} > C_{j,1} > C_{i,d_i} > C_{j,d_j}$, i.e.
\begin{align*}
\{C_{i,1},\ \ \dots,\ \ C_{i,d_i-p}&,\ \  \overbrace{C_{i,d_i-p+1},\ \ \dots,\ \ C_{i,d_i}}^{p} \} \\
&\{C_{j,1}, \ \ \  \dots\dots,\ \ \ C_{j,p},\ \ \ \ C_{j,p+1},\ \ \dots, \ \  C_{j,d_j}\}
\end{align*}
with $C_{j,1} = C_{i,d_i} + 2p-1$ and $d_j > p$, then we change the coordinates of $\mathcal{T}_i$
and $\mathcal{T}_j$ into:
\[\boxed{\begin{aligned}
\mathcal{T}_i' &: (*,\dots,\ *,\ \overbrace{k_i+1,\ \dots,\ k_i+p}^{p} ) \\
\mathcal{T}_j' &: \ \ \ \ \ \ \ \ \ \ (k_j-1, \ \dots,\ k_j-p,\ *,\ \ \dots, \ \  *).
\end{aligned}}\]
where the entries marked by $*$ remain unchanged.

\item[(c)] If $C_{j,1} > C_{i,1} > C_{j,d_j}$, then since $k_i \geq k_j$ one also have
$C_{j,1} > C_{i,1} \geq C_{i,d_i} > C_{j,d_j}$
i.e.
\begin{align*}
\{C_{i,1}, \ \ \dots ,\ \ &C_{i,d_i}\} \\
\{\underbrace{C_{j,1},\ \ \ \ \ \ \dots, \ \ \ \ \ \   C_{j,q}}_{q},&\ \ \ \ \ \ C_{j,q+1}, \ \ \dots, \ \ C_{j,d_j}\}
\end{align*}
with $C_{j,1} = C_{i,d_i} + 2q-1$, then we change the coordinates of $\mathcal{T}_i$ and $\mathcal{T}_j$ into:
\[\boxed {\begin{aligned}
\mathcal{T}_i'&: \ \ \ \ \ \ \ \ \ \ \ \ \ \ \ \ \  (k_i + (q-d_0+1),\ \dots,\ k_i + (q-1),\ k_i + q) \\
\mathcal{T}_j'&: (\underbrace{*,\ \dots,\ *,\ k_j - (q-d_0+1),\ \dots,\ k_j - (q-1),\ k_j -q}_{q},\ *, \ \ \dots, \ \ *),
\end{aligned}}\]
where the entries marked by $*$ remain unchanged.
\end{itemize}
In the above three cases, we only demonstrate the situation that $\mathcal{C}_i$ is in the first row and $\mathcal{C}_j$ is in the second row. The rule is the same when $\mathcal{C}_j$ is in the first row while $\mathcal{C}_i$ is in the second row.

After running through all pairs of
linked chains, $V_{\tau}$ is defined as the $K$-type with highest weight $\tau$ given by (a $W$-conjugate of)
$\bigcup_{i=0}^m \mathcal{T}_i'$.
\end{algorithm}

\begin{example}  \label{eg:spinlowest}
Consider $(\lambda,-s\lambda) = \begin{aligned}
\{10 && && 8\} && && \{6\} && && \{4\} \\
&& \{9 && && 7 && && 5 && && 3 && && 1\}
\end{aligned}$. Then  the lowest $K$-type of $J(\lambda,-s\lambda)$ is
\begin{align*}
(9 && && 9) && && (6) && && (4) \\
&& (5 && && 5 && && 5 && && 5 && && 5)
\end{align*}

To compute $V_{\tau}$, let us label the chains so that \eqref{eq:order} holds:
$$
\mathcal{T}_0=(9 \quad 9),\quad \mathcal{T}_1=(6), \quad \mathcal{T}_2=(5 \quad 5\quad  5 \quad 5 \quad 5),\quad \mathcal{T}_3=(4).
$$
Then we apply (a) to the pair $\mathcal{T}_2$, $\mathcal{T}_3$, apply (b) to the pair $\mathcal{T}_0$, $\mathcal{T}_2$, and apply (c) to the pair $\mathcal{T}_1$, $\mathcal{T}_2$. This gives us
\begin{align*}
(9 && && 10) && && (8) && && (2) \\
&& (4 && && 3 && && 5 && && 7 && && 5).
\end{align*}
Thus $\tau = (10,9,8,7,5,5,4,3,2)$.
\end{example}

\begin{theorem}
Let $J(\lambda,-s\lambda)$ be a unitary module of $GL(n)$ in Theorem \ref{unitary}, and $V_{\tau}$ be obtained
by Algorithm \ref{alg:spinlkt}. Then $[J(\lambda,-s\lambda):V_{\tau}] > 0$.
\end{theorem}
\begin{proof}
Let $\displaystyle J(\lambda,-s\lambda) = {\rm Ind}_{\prod_{i=0}^m GL(a_i)}^{GL(n)}(\bigotimes_{i=0}^m V_{(k_i,\dots,k_i)})$.
By rearranging the Levi factors, one can assume the chains $\mathcal{C}_0$, $\dots$, $\mathcal{C}_m$ satisfy Equation \eqref{eq:order}.
We are interested in studying
\begin{align*}
\left[{\rm Ind}_{\prod_{i=0}^m GL(a_i)}^{GL(n)}(\bigotimes_{i=0}^m V_{(k_i,\dots,k_i)}): V_{\tau}\right]
&= \left[\bigotimes_{i=0}^m V_{(k_i,\dots,k_i)}: V_{\tau}|_{\prod_{i=0}^m GL(a_i)}\right] \\
&= \left[\bigotimes_{i=0}^m V_{(k_i+t,\dots,k_i+t)}: V_{\tau}|_{\prod_{i=0}^m GL(a_i)} \otimes \bigotimes_{i=0}^m V_{(t,\dots,t)}\right] \\
&= \left[\bigotimes_{i=0}^m V_{(k_i+t,\dots,k_i+t)}: V_{\tau}|_{\prod_{i=0}^m GL(a_i)} \otimes V_{(t,\dots,t)}|_{\prod_{i=1}^m GL(a_i)}\right] \\
&= \left[\bigotimes_{i=0}^m V_{(k_i+t,\dots,k_i+t)}: V_{\tau+(t,\dots,t)}|_{\prod_{i=0}^m GL(a_i)}\right] \\
&= \left[{\rm Ind}_{\prod_{i=0}^m GL(a_i)}^{GL(n)}(\bigotimes_{i=0}^m V_{(k_i+t,\dots,k_i+t)}): V_{\tau+(t,\dots,t)}\right]
\end{align*}
So we can assume $k_i > 0$ for all $i$ without loss of generality.

We prove the theorem by induction on the number of Levi components.
The theorem obviously holds when there is only one Levi component -- the
irreducible module is a unitary character of $GL(n)$. Now suppose we have
the hypothesis holds when there are $m$ Levi factors, i.e.
$$\left[{\rm Ind}_{\prod_{i=0}^{m-1} GL(a_i)}^{GL(n')}(\bigotimes_{i=0}^{m-1} V_{(k_i,\dots,k_i)}) : V_{\tau_{m-1}}\right] > 0,$$
where $n' = n - a_m$, and $\tau_{m-1}$ is obtained by applying Algorithm \ref{alg:spinlkt} on
$\bigcup_{i=0}^{m-1} \mathcal{C}_i$. Suppose now $\tau_{m}$ is obtained by applying
Algorithm \ref{alg:spinlkt} on $\bigcup_{i=0}^m \mathcal{C}_i$. Then
\begin{align*}
&\ \left[{\rm Ind}_{\prod_{i=0}^{m} GL(a_i)}^{GL(n)}(\bigotimes_{i=0}^{m} V_{(k_i,\dots,k_i)}) : V_{\tau_{m}}\right] \\
= &\ \left[{\rm Ind}_{GL(n') \times GL(a_m)}^{GL(n)}\left({\rm Ind}_{\prod_{i=0}^{m-1} GL(a_i)}^{GL(n')}(\bigotimes_{i=0}^{m-1} V_{(k_i,\dots,k_i)}) \otimes V_{(k_m,\dots,k_m)}\right) : V_{\tau_{m}}\right] \\
\geq &\ \left[{\rm Ind}_{GL(n') \times GL(a_m)}^{GL(n)}(V_{\tau_{m-1}} \otimes V_{(k_m,\dots,k_m)}) : V_{\tau_{m}}\right] \\
=&\ c_{\tau_{m-1}, (k_m,\dots,k_m)}^{\tau_{m}}
\end{align*}
 Here $c_{\mu,\nu}^{\lambda}$ is the Littlewood-Richardson coefficient, and the last step uses Theorem 9.2.3 of \cite{GW}.

Suppose $\tau_{m-1} = \bigcup_{i=0}^{m-1} \mathcal{T}_i''$. Here these $\mathcal{T}_i''$ are obtained by
applying Algorithm \ref{alg:spinlkt} on $\mathcal{C}_0$, $\dots$, $\mathcal{C}_{m-1}$.
Then $\tau_m$ is obtained from applying Algorithm \ref{alg:spinlkt} on $\mathcal{T}_i''$ and $\mathcal{T}_m = (k_m,\dots,k_m)$
for all linked $\mathcal{C}_i$ and $\mathcal{C}_m$.
More precisely, by applying Rules (a) -- (c) in Algorithm \ref{alg:spinlkt}, $\tau_m$ is obtained from $\tau_{m-1}$ by the following:
\begin{itemize}
\item[(i)] Construct a new partition $\tau_{m-1} \cup (k_m,\dots,k_m)$.
\item[(ii)] For each linked $\mathcal{C}_i$ and $\mathcal{C}_m$, add $(0,\dots,0, A, A-1,\dots,a+1,a,0,\dots,0)$ on the rows of
$\tau_{m-1}$ corresponding to $\mathcal{T}_i''$, and subtract $(0,\dots,0, A, A-1,\dots,a+1,a,0,\dots,0)$ on the corresponding rows of $(k_m,\dots,k_m)$.
\item[(iii)] $\tau_m$ is obtained by going through (ii) for all $\mathcal{C}_i$ linked with $\mathcal{C}_m$.
\end{itemize}
By the above construction of $\tau_m$, it follows from the Littlewood-Richardson Rule as stated on page 420 of \cite{GW} that
\begin{equation}\label{LR-geq1}
c_{\tau_{m-1}, (k_m,\dots,k_m)}^{\tau_{m}} \geq 1.
\end{equation}
Indeed,
it suffices to find \emph{one} \emph{L-R skew tableaux} of shape $\tau_m/\tau_{m-1}$ and weight
$$
(\underbrace{k_m,\dots, k_m}_{d_m})
$$
in the sense of Definition 9.3.17 of \cite{GW}. Recall that $d_m$ is the number of entries of the chain $\mathcal{C}_m$.

To do so, we first describe the Ferrers diagram $\tau_m/\tau_{m-1}$. Suppose
$\mathcal{C}_{i_1}$, $\dots$, $\mathcal{C}_{i_l}$ are linked to $\mathcal{C}_m$
with $i_1 > \dots > i_l$. By Step (ii) of the above algorithm, we add $(A_j, A_j - 1, \dots, a_j +1, a_j)$
to the rows in $\tau_{m-1}$ corresponding to the chains $\mathcal{C}_{i_j}$.
Note that by our ordering of the chains, we must have
$$A_l > \dots > a_l > A_{l-1} > \dots > a_{l-1} > \dots > A_1 > \dots > a_1.$$
The rows of the Ferrers diagram $\tau_m/\tau_{m-1}$ have lengths
\begin{equation} \label{eq-skew}
\underbrace{A_1, \dots, a_1}_{:= \mathcal{R}_1}; \cdots; \underbrace{A_l, \dots, a_l}_{:= \mathcal{R}_l}; \underbrace{k_m,\dots, k_m; (k_m - a_1), \dots, (k_m - A_1); \cdots;  (k_m - a_l), \dots, (k_m - A_l)}_{:= \mathcal{R}_{l+1}}
\end{equation}
with $\sum_{j=1}^{l+1} |\mathcal{R}_j| = d_m$, where $|\mathcal{R}_j|$ is the number of entries in $\mathcal{R}_j$.

Now we fill in the entries on each row of $\tau_m/\tau_{m-1}$ as follows. Consider the standard Young tableau $T$ whose row sizes are $(\underbrace{k_m,\dots, k_m}_{d_m})$ and the
entries of the $i$-th row of $T$ are all equal to $i$. Now let a sequence of subtableaux
of $T$ given by
$$T_1 \subset T_2 \subset \dots \subset T_l \subset T_{l+1} := T$$
such that for each $1 \leq j \leq l$, $T_j$ has the shape of the form
$$
A_j > \dots > a_j > \dots > A_1 > \dots > a_1.
$$
Consider the skew tableau $T_j/T_{j-1}$ for $1 \leq j \leq l+1$ (where we take $T_0$ to be the empty tableau),
then the column sizes of $T_j/T_{j-1}$ is the same as the parametrization for the tableau $\mathcal{R}_{j}$ marked
in \eqref{eq-skew}.

For each $1 \leq j \leq l+1$, fill in the rows of the Ferrers diagram $\tau_m/\tau_{m-1}$
corresponding to $\mathcal{R}_j$ in \eqref{eq-skew} by filling the $t$-th row of $\mathcal{R}_j$
with the $t$-th entries on each column of $T_j/T_{j-1}$ counting from the top in ascending
order. This will give us a \emph{semi-standard skew tableau} of shape $\tau_m/\tau_{m-1}$ and weight $(\underbrace{k_m,\dots, k_m}_{d_m})$ (see Definition 9.3.16 of \cite{GW}), whose row word is a \emph{reverse lattice word} by Definition 9.3.17 of \cite{GW}. To sum up, it is a desired L-R tableau and \eqref{LR-geq1} follows.
%
%
%
%
%
%
%
%
%
%
%
%
%
%
%
%
\end{proof}

\begin{proposition}\label{prop-spin-lowest}
Let $J(\lambda,-s\lambda)$ be a unitary module of $GL(n)$ in Theorem \ref{unitary}, and
$V_{\tau}$ be the $K$-type obtained by Algorithm \ref{alg:spinlkt}.
Then $\tau$ satisfies
$$\{\tau - \rho\} = 2\lambda - \rho.$$
Consequently, $V_{\tau}$ is a spin-lowest $K$-type of $J(\lambda,-s\lambda)$ by Equation \eqref{spin=2lambda}.
\end{proposition}
\begin{proof}
We prove the proposition by induction on the number of chains in $(\lambda,-s\lambda) = \bigcup_{i=0}^m \mathcal{C}_i$, where the chains are arranged so that Equation \eqref{eq:order} holds.
Suppose that the proposition holds for $\bigcup_{i=0}^{m-1} \mathcal{C}_i$. There are two possibilities when adding $\mathcal{C}_m$:

\begin{itemize}
\item There exists $\mathcal{C}_i$ such that $\mathcal{C}_i$ and $\mathcal{C}_m$ is related by Rule (a) in Algorithm \ref{alg:spinlkt}:
\begin{align*}
\{ \ \ \ \ \ \ \ \ \ \ \ \ \ \ \ \ \ \ \ \ \ \ \mathcal{C}_i\ \ \ \ \ \ \ &\ \ \ \  \ \ \ \ \ \ \ \ \ \ \ \} \\
&\{\ \ \mathcal{C}_{m}\ \ \}
\end{align*}
\item There exist $\mathcal{C}_j$ and $\mathcal{C}_{r}$, $\dots$,  $\mathcal{C}_{m-1}$,
such that $\mathcal{C}_j$ and $\mathcal{C}_m$ are related by Rule (b), and
$\mathcal{C}_{l}$, $r \leq l \leq m-1$ and $\mathcal{C}_m$ are related by Rule (c) in Algorithm \ref{alg:spinlkt}:
\begin{align*}
\{\ \ \ \ \ &\mathcal{C}_j\ \ \ \ \ \} \ \ \ \ \  \{\ \ \mathcal{C}_{r}\ \ \} \ \ \ \dots \ \ \ \{\ \ \mathcal{C}_{m-1}\ \ \} \\
&\{ \ \ \ \ \ \ \ \ \ \ \ \ \ \ \ \ \ \ \ \ \ \ \mathcal{C}_m\ \ \ \ \ \ \ \ \ \ \  \ \ \ \ \ \ \ \ \ \ \ \}
\end{align*}
\end{itemize}
We will only study the second case, and the proof of the first case is simpler. Suppose the
chains in the second case are interlaced in the following fashion:
\begin{equation} \label{eq:interlaced}
\begin{aligned}
\{\ \ \ \ &\ \mathcal{C}_{j}\ \ \ \ \ \ \ \ \ \}   &&\ \ \{\overbrace{\ \ \mathcal{C}_{r}\ \ }^{d_{r}}\} \ \ \ \ \ \ \ \ \ \ \cdots\cdots &&\ \ \{\overbrace{\ \ \ \mathcal{C}_{m-1}\ \ \ }^{d_{m-1}}\}\\
\{&\underbrace{C_{m,1}, \cdots}_{p}\ \ \underbrace{\cdots\cdots }_{a_r} &&\underbrace{\cdots\ \ \cdots }_{d_{r}} \ \ \underbrace{\cdots\cdots}_{a_{r+1}} \ \ \cdots\cdots &&\underbrace{\cdots \ \ \ \ \cdots  }_{d_{m-1}},\ \underbrace{\cdots,\ C_{m,d_m}}_{a_m}\}
\end{aligned}
\end{equation}
for some $j+1 \leq r \leq m-1$, and the chains $\mathcal{C}_{j+1}$, $\dots$, $\mathcal{C}_{r-1}$---which have not been shown in \eqref{eq:interlaced}---are linked with $\mathcal{C}_j$ under Rule (a) of Algorithm \ref{alg:spinlkt}.

To simplify the calculations below, we introduce the notation
$$(a)^{\epsilon}_d:=\underbrace{a, a+\epsilon, \dots, a+(d-1)\epsilon}_d.$$
Then $2\lambda$ is equal to the entries in Equation \eqref{eq:interlaced}.
Since the values of the adjacent entries within the same chain differ by $2$,
and the values of the interlaced entries differ by $1$, one can calculate
$2\lambda - \rho$ up to a translation by a constant on all coordinates as follows:
\begin{equation} \label{2lambda}
\begin{aligned}
\{\cdots&\ (A_{r-1})^0_p\}   &&\ \ \{(A_{r})^0_{d_{r}}\} \ \ \ \ \ \ \ \ \ \ \cdots\cdots &&\ \ \{(A_{m-1})^0_{d_{m-1}}\}\\
\dots\ \ \{&(A_{r-1})^0_p\ \ (A_r)^{-1}_{a_r} &&(A_{r})^0_{d_{r}}\ \ (A_{r})^{-1}_{a_{r+1}} \ \ \cdots\cdots &&(A_{m-1})^0_{d_{m-1}}\ \ (A_{m-1})^{-1}_{a_m}\}
\end{aligned}
\end{equation}
where $\displaystyle A_x := \sum_{l=x}^{m-1} a_{l+1}$ for $r-1 \leq x\leq m-1$
(note that the smallest entry of \eqref{2lambda} is $1$, appearing at the rightmost entry of the bottom chain).

\smallskip

On the other hand, the calculation in Algorithm \ref{alg:spinlkt} gives $\tau$ as follows:
$$\begin{aligned}
(\cdots & \ \ (k_j)^{0}_{p})  \ \ \ \ \ \ \ \ \ \ \ \ \ \  (k_{r})^{0}_{d_{r}} \ \ \ \ \ \ \ \cdots\ \ \ \ \ \ \ (k_{m-1})^{0}_{d_{m-1}}\\
\dots \ \ (&(k_m)^{0}_{p}\ \ (k_m)^{0}_{a_r}\ \ (k_m)^{0}_{d_{r}}\ (k_m)^{0}_{a_{r+1}} \  \cdots\ (k_m)^{0}_{d_{m-1}}(k_m)^{0}_{a_m})
\end{aligned} = \bigcup_{i=0}^m \mathcal{T}_i \ \longrightarrow\ \
\bigcup_{i=0}^m \mathcal{T}_i' = \tau,$$
where $\bigcup_{i=0}^m \mathcal{T}_i'$ is given by
\begin{equation} \label{tau}
\begin{aligned}
(\ \cdots   &\ (k_j+1)^{1}_{p})   \ \ \ \ \ \ \ \ \ \ \   (k_{r}+(q_r-d_{r}+1))^1_{d_{r}} \ \ \ \ \ \ \cdots\ \ \ \ \ \ \ (k_{m-1}+(q_{m-1}-d_{m-1}+1))^1_{d_{m-1}}\\
\dots \ \ (&(k_m-1)^{-1}_{p} (k_m)^{0}_{a_r} (k_m-(q_r-d_{r}+1))^{-1}_{d_{r}}(k_m)^{0}_{a_{r+1}}  \cdots(k_m-(q_{m-1}-d_{m-1}+1))^{-1}_{d_{m-1}}(k_m)^{0}_{a_m})
\end{aligned}
\end{equation}
and $q_i$ are obtained by Rule (c) of Algorithm \ref{alg:spinlkt}. For instance, $q_r=p+a_r+d_{r}$. Note that
$$
k_j-(d_j-1)=k_{r}+(d_{r}-1)+2a_r+2.
$$
Therefore,
$$
k_j-d_j=k_{r}+d_{r}+2a_r.
$$
From this, one deduces easily that $k_j\geq k_{r}+q_r+1$. Thus it makes sense to talk about the interval $[k_{r}+q_r+1, k_j]$.

Before we proceed, we pay closer attention to the coordinates of $\mathcal{T}_{j}'$, which is the
leftmost chain on the top row of Equation \eqref{tau}. More precisely, it consists of three parts:
\begin{itemize}
\item[(i)] As mentioned in the paragraph after Equation \eqref{eq:interlaced}, by applying Rule (a) of Algorithm \ref{alg:spinlkt}
between $\mathcal{C}_j$ and each of $\mathcal{C}_{j+1}$, $\dots$, $\mathcal{C}_{r-1}$,
one can check that
$$\bigcup_{i=j+1}^{r-1}\mathcal{T}_{i}' \subset [k_{r}+q_r+1, k_j].$$
Suppose there are $\delta \geq 0$ coordinates in $\bigcup_{i=j+1}^{r-1}\mathcal{T}_{i}'$,
then there will be exactly $\delta$ coordinates in $\mathcal{T}_{j}'$ having coordinates strictly
greater than $k_j + p$.
\item[(ii)] By applying Algorithm \eqref{alg:spinlkt} to $\mathcal{C}_j$ and $\mathcal{C}_m$,
we have $p$ coordinates $(k_j+1)^1_p$ in $\mathcal{T}_j'$ as in Equation \eqref{tau}.
\item[(iii)] The other coordinates of $\mathcal{T}_j'$ are either equal to $k_j$, or smaller
than $k_j$ if they are linked with $\mathcal{C}_t$ with $t < j$.
\end{itemize}
In conclusion, the coordinates of $\mathcal{T}_j'$ are given by $(\overbrace{\sharp\ \dots\ \sharp}^{\delta} ; (k_j+1)^1_p; \overbrace{\flat\ \dots\ \flat}^{d_j - \delta - p})$,
where $\sharp\ \dots\ \sharp$ has coordinates greater than $k_j+p$, and
$\flat\ \dots\ \flat$ has coordinates smaller than $k_j + 1$.

We now arrange the coordinates of $\bigcup_{i = j}^{m} \mathcal{T}_i'$ in Equation \eqref{tau} as follows:
\begin{align*}
&\overbrace{\sharp\ \dots\ \sharp}^{\delta} > \overbrace{(k_j+1)^{1}_{p}}^{p} > \overbrace{\flat \cdots \flat}^{d_j-p-\delta} > \overbrace{\bigcup_{i=j+1}^{r-1} \mathcal{T}_i'}^{\delta} > \mathcal{T}_r' > \dots > \mathcal{T}_{m-1}' > (k_m)^{0}_{a_r} = \dots = (k_m)^{0}_{a_m} \\
&  > (k_m-1)^{-1}_{p} > (k_m-(q_r-d_{r}+1))^{-1}_{d_{r}} > \dots > (k_m-(q_{m-1}-d_{m-1}+1))^{-1}_{d_{m-1}}
\end{align*}
Here elements in the blocks $\mathcal{T}'_r, \dots, \mathcal{T}'_{m-1}$ are still kept in the increasing manner. Note that
if $x < y$, then $\mathcal{T}_x' > \mathcal{T}_y'$ in terms of their coordinates.

We index the coordinates of $\tau$ shown in Equation \eqref{tau} using the above ordering, with the smallest coordinate indexed by $1$:
\begin{equation} \label{rho}
\begin{aligned}
(\dots &\ (d_m+D_r+ d_j - p+1)^1_p) \ \ ((d_m+D_{r+1}+1)^1_{d_{r}}) \ \ \ \  \ \cdots\ \ \ \ \ \ ((d_m+1)^1_{d_{m-1}})\\
(&(D_r+p)^{-1}_p  (D_r + p + 1)^1_{a_r}(D_r)^{-1}_{d_{r}} (D_r + p + a_r + 1)^1_{a_{r+1}} \cdots (D_{m-1})^{-1}_{d_{m-1}} (D_r + p + \sum_{l=r}^{m-1} a_l + 1)^1_{a_m}),
\end{aligned}
\end{equation}
where $\displaystyle D_x := \sum_{l = x}^{m-1} d_{l}$ for $r\leq x\leq m-1$. Note that the coordinates of the last row read as
\begin{align*}
(D_r + p,  \dots, 2, 1)&=((D_r+p)^{-1}_p;\ (D_r)^{-1}_{d_{r}};\ \dots;\ (D_{m-1})^{-1}_{d_{m-1}}),\\
(D_r + p + 1, \dots, d_m-1, d_m)&=\\
((D_r + p + 1)^1_{a_r};\ \dots ;&\ (D_r + p + \sum_{l=r}^{x-1} a_l + 1)^1_{a_x};\ \dots;\  (D_r + p + \sum_{l=r}^{m-1} a_l + 1)^1_{a_m}).
\end{align*}

\smallskip

Up to a translation of a constant of all coordinates, the difference between Equation \eqref{tau} and \eqref{rho} gives (a $W$-conjugate of) $\{\tau - \rho\}$, which is of the form:
\begin{equation} \label{taurho}
\begin{aligned}
(\cdots &\  (\beta_j)^0_p)  \ \ \ \ \ \ \ \ \ \ \  (\beta_r)^0_{d_{r}} \ \ \ \ \ \ \cdots\ \ \ \ \ \ (\beta_{m-1})^0_{d_{m-1}}\\
(&(\alpha_j)^0_p\ \ {\bf *\ *\ *\ }(\alpha_r)^0_{d_{r}}\ {\bf *\ *\ *}\  \cdots\ (\alpha_{m-1})^0_{d_{m-1}}\ {\bf *\ *\ *})
\end{aligned}
\end{equation}
Our goal is to show \eqref{2lambda} and \eqref{taurho} are equal up to a translation of a constant of all coordinates. So we need to show the following:
\begin{itemize}
\item[(i)] $\alpha_j = \beta_j$: We need to show
$$k_m - 1 - (D_r + p) = k_j + 1 - (d_m + D_r + d_j-p + 1).$$
In fact, we have
\begin{align*}
C_{m,1} &= C_{j,d_j} + 2p - 1 \\
k_m +         (d_m - 1) &= k_j - (d_j-1) + 2p - 1 \\
k_m - p  -1    &=   k_j  - d_j + p  - d_m   \\
k_m - 1 - (D_r + p)   &=   k_j + 1  - (d_m + D_r + d_j -p + 1)
\end{align*}
as required.

\item[(ii)] $\alpha_x = \beta_x$ for all $r \leq x \leq m-1$: This is the same as showing
$$k_m - (q_x - d_{x} + 1) - D_x = k_{x} + (q_x - d_{x} + 1) - (d_m + D_{x+1}+1).$$
As in (i), we consider
\begin{align*}
C_{m,1} &= C_{x,d_{x}} + 2q_x - 1 \\
k_m +         (d_m - 1) &= k_{x} - (d_{x}-1) + 2q_x - 1 \\
k_m - q_x + d_{x} -1    &=   k_{x}  + q_x  - d_m   \\
k_m - q_x + d_{x} -1 - D_x + D_{x+1} + d_{x}   &=   k_{x}  + (q_x+1)  - (d_m + 1)   \\
k_m - q_x + d_{x} -1 - D_x   &=   k_{x}  + (q_x-d_{1}+1)  - (d_m + D_{x+1} + 1)
\end{align*}
as we wish to show.\\

\item[(iii)] $\alpha_j - \alpha_x = A_{r-1} - A_{x}$ for all $r \leq x \leq m-1$: In other words, we need to show
$$[(k_m - 1) - (D_r +p)] - [(k_m - (q_x-d_{x}+1)) - D_x] = A_{r-1} - A_{x} = a_r + \dots + a_x$$
Indeed, by looking at Equation \eqref{eq:interlaced} and applying Rule (c) of Algorithm \ref{alg:spinlkt}, one gets
\begin{align*}
p + (a_r + \dots + a_x) + (d_r + \dots + d_x) &= q_x \\
q_x - p  &= (A_{r-1} - A_{x}) + (D_r - D_{x+1})  \\
(k_m - 1) - (k_m - 1) + q_x - p - D_r + D_{x+1} &= A_{r-1} - A_{x} \\
[(k_m - 1) - (D_r +p)] - (k_m -1) + q_x + (D_x - d_x) &= A_{r-1} - A_{x} \\
[(k_m - 1) - (D_r +p)] - [(k_m - (q_x-d_x+1)) - D_x] &= A_{r-1} - A_{x}
\end{align*}
so the result follows.\\

\item[(iv)] Collecting the $*\ *\ *$ entries of Equation \eqref{taurho} consecutively from left to right gives
$$\underbrace{\alpha_j,\dots,\alpha_r+1}_{a_r};\ \cdots\cdots;\ \underbrace{\alpha_x,\dots,\alpha_{x+1}+1}_{a_{x+1}};\ \cdots\cdots;\ \underbrace{\alpha_{m-1},\dots,\alpha_{m-1} - (a_m-1)}_{a_m}$$
In order for the above expression to make sense, one needs $\alpha_x - \alpha_{x+1} = a_x$ for all $r\leq x \leq m-1$ for instance. This is indeed the case,
since $\alpha_x - \alpha_{x+1} = A_{x} - A_{x+1}$ by (iii), and the latter is equal to $a_{x+1}$ by the definition of $A_x$ for $r-1 \leq x\leq m-1$. So it suffices to check $\displaystyle k_m - (D_r + p + \sum_{l=r}^{x} a_l + 1) = \alpha_x.$

To see it is the case, one can check that the leftmost entry of the second row of Equation \eqref{taurho} is equal to
\begin{align*}
\alpha_j &= k_m - 1 - (D_r + p) \\
\alpha_x + A_{r-1} - A_{x} &= k_m - (D_r + p + 1)   \ \ \ \ \ \ \ \ \ \text{(by (iii))} \\
\alpha_x + \sum_{l = r}^{x} a_l &= k_m - (D_r + p + 1) \\
\alpha_x &= k_m - (D_r + p + \sum_{l=r}^{x} a_l + 1)
\end{align*}
as follows.
\end{itemize}

\smallskip

Combining (i) -- (iv), Equation \eqref{taurho} can be rewritten as
\begin{align*}
(\cdots&\ (\alpha_j)^0_p)   &&\ \ \ ((\alpha_r)^0_{d_r}) \ \ \ \ \ \ \ \ \ \ \cdots\cdots &&\ \ ((\alpha_{m-1})^0_{d_{m-1}})\\
(&(\alpha_j)^0_p\ \ \ \ \ \ \ (\alpha_j)^{-1}_{a_r} &&(\alpha_r)^0_{d_r}\ \ \ \ \  (\alpha_r)^{-1}_{a_{r+1}} \ \ \ \ \cdots\cdots &&(\alpha_{m-1})^0_{d_{m-1}}\ \ (\alpha_{m-1})^{-1}_{a_m})
,\end{align*}
whose coordinates are in descending order from left to right. So it is equal to
$\{\tau - \rho\}$ up to a translation of a constant. Moreover, by comparing it with Equation \eqref{2lambda},
we have shown that all coordinates of $2\lambda - \rho$ and $\{\tau - \rho\}$
differ by a constant (note that the other coordinates on the left of $\mathcal{C}_j$ are taken care of by induction hypothesis).
To see they are exactly equal to each other, we calculate the {\it true} values of $A_{m-1}$ and $\alpha_{m-1}$ in
$2\lambda - \rho$ and $\tau$ respectively on the entry marked by $\circledast$ below:
\begin{align*}
\{\dots, &\ \ *, \dots ,* \}   &&\ \ \ \{*,\dots,*\} \ \ \ \ \ \ \cdots\  &&\ \ \{*,\dots, *\}\\
\{&*,\dots, *;\ \ \ \ *,\dots,*;\ &&*, \dots,*;\ \ \ *,\dots,*;\ \ \ \cdots;\ &&*,\dots,\circledast;\ \ \underbrace{*,\dots,*}_{a_m}\}
\end{align*}
For $2\lambda - \rho$, $\circledast$ takes the value
$$C_{m, d_m - a_m} - \rho_{a_m + 2},$$
where $\rho = (\rho_n, \dots, \rho_2, \rho_1)$ with $\rho_i = \rho_1 + (i-1)$. So it can be simplified as
\begin{align*}
C_{m,d_m - a_m} - \rho_{a_m + 2} &=
k_m - (d_m - 1) + 2a_m - \rho_{a_m + 2} \\
&= k_m - d_m + 1 + 2a_m - \rho_1 - (a_m +1)\\
&= k_m - d_m + a_m - \rho_1
\end{align*}
On the other hand, for $\{\tau - \rho\}$, $\circledast$ takes the value
$$k_m - q_{m-1} - \rho_{1}$$
(Recall that we had $\alpha_{m-1} = k_m - q_{m-1} - 1$ for $\circledast$ in our previous calculation).

By looking at Equation \eqref{eq:interlaced} and applying Rule (c) of Algorithm \ref{alg:spinlkt} again, one has $q_{m-1} = d_m - a_m$, hence $2\lambda - \rho$
and $\{\tau - \rho\}$ takes the same value on the $\circledast$ coordinate.
Since we have seen that their coordinates differ by the same constant, one can conclude that $2\lambda - \rho = \{\tau - \rho\}$.
\end{proof}

\begin{example}
For the the interlaced chain in Example \ref{eg:spinlowest}, the translate of $2\lambda - \rho$
in Equation \eqref{2lambda} is equal to
\begin{align*}
\ &\begin{aligned}
\{10-8 && && 8-6\} && && \{6-4\} && && \{4-2\} \\
&& \{9-7 && && 7-5 && && 5-3 && && 3-1 && && 1-0\}
\end{aligned} \\
=\
 &\begin{aligned}
\{2 && && 2\} && && \{2\} && && \{2\} \\
&& \{2 && && 2 && && 2 && && 2 && && 1\}
\end{aligned}.
\end{align*}
Also, the translate of $\tau - \rho$ in Equation \eqref{taurho} is given by:
\begin{align*}
\ &\begin{aligned}
(9-8 && && 10-9) && && (8-7) && && (2-1) \\
&& (4-3 && && 3-2 && && 5-4 && && 7-6 && && 5-5)
\end{aligned}\\
=\ &\begin{aligned}
(1 && && 1) && && (1) && && (1) \\
&& (1 && && 1 && && 1 && && 1 && && 0)
\end{aligned}
\end{align*}
Hence their coordinates differ by the same constant $1$. To see $2\lambda - \rho$
and $\{\tau - \rho\}$ are equal, where $\rho = (4,3,2,1,0,-1,-2,-3,-4)$,
one can look at the {\it true} values
of them for the rightmost entry of the bottom chain:
\[
2\lambda - \rho:\ 1 - \rho_1 = 1 - (-4) = 5;\ \ \ \ \ \tau - \rho:\ 5 - \rho_5 = 5 - 0 = 5.
\]
Hence $2\lambda - \rho = \{\tau - \rho\} = (6,6,6,6,6,6,6,6,5)$, and the unique $\widetilde{K}$-type in the Dirac cohomology of
the corresponding unitary module is $V_{(6,6,6,6,6,6,6,6,5)}$.
\end{example}

\section{Scattered Representations of $SL(n)$}
It is easy to parametrize irreducible unitary representations of $SL(n)$ using the parametrization for $GL(n)$. In such cases, we
impose the condition on $\lambda$ such that the sum of the coordinates is equal
to $0$. In other words, for each possible regular, half-integral infinitesimal
character $\lambda$ for $SL(n)$, one can shift the coordinates by a suitable scalar,
so that it corresponds to an infinitesimal character $\lambda'$ of $GL(n)$ whose smallest coordinate
is equal to $1/2$.

Therefore, the irreducible unitary representations of $SL(n)$ are
parametrized by chains with $n$ coordinates whose smallest coordinate is equal to $1$.

The following proposition characterizes which of these representations are scattered in the sense of
Section \ref{scattered}:
\begin{proposition} \label{prop:scattered}
Let $\pi := J(\lambda,-s\lambda)$ be an irreducible unitary representation of $SL(n)$ such that $\lambda$ is dominant and half-integral. Then
$\pi$ is a scattered representation if and only if the translated Zhelobenko parameter
$(\lambda',-s\lambda')$ can be expressed as
a union of interlaced chains with smallest coordinate equal to $1$.
\end{proposition}
\begin{proof}
By the arguments in Section \ref{scattered}, one only needs to check
that $s \in W$ involves all simple reflections in its reduced expression
if and only if $(\lambda',-s\lambda') = \bigcup_{i=0}^m \mathcal{C}_i$ are interlaced.
Indeed, $s \in W$ can be read from $\bigcup_{i=0}^m \mathcal{C}_i$ as follows:
label the entries of $\bigcup_{i=0}^m \mathcal{C}_i$ in descending order, e.g.
$$\bigcup_{i=0}^m \mathcal{C}_i = \begin{aligned}
&\ \ \{p_{k+1},\ \ \dots \}\ \  \cdots \\
\{p_1,\ \ p_2, \ \ \dots,\ \ &p_k,  \ \ \ p_{k+2},\ \ \dots \}\ \ \cdots
\end{aligned}$$
with $p_1 > p_2 > \dots > p_n$, then we `flip' the entries of each chain $\mathcal{C}_i$ by
$\{C_{i,1},\dots,C_{i,d_i}\}$ $\rightarrow$ $\{C_{i,d_i},\dots,C_{i,1}\}$.
Suppose we have
$$\begin{aligned}
\begin{aligned}
&\ \ \{p_{s_{k+1}},\ \ \dots \}\ \  \cdots \\
\{p_{s_1},\ \ p_{s_2}, \ \ \dots,\ \ &p_{s_k},  \ \ \ p_{s_{k+2}},\ \ \dots \}\ \ \cdots
\end{aligned}
\end{aligned}$$
after flipping each chain, then $s \in S_n$ is obtained by $s = \begin{pmatrix}1 & 2 & \dots & n \\ s_1 & s_2 & \dots & s_n \end{pmatrix}$
(see Example \ref{eg:interlaced}).

Define the equivalence class of interlaced chains by letting $\mathcal{C}_i \sim \mathcal{C}_j$
iff $i = j$, or $\mathcal{C}_i, \mathcal{C}_j$ are interlaced.
So we have a partition of $\{p_1, \dots, p_n\}$ by the entries of chains in the same equivalence class.
It is not hard to check that the entries on each partition have consecutive indices, i.e.
$$\mathcal{E}_i = \{p_{a_i}, p_{a_i + 1}, \dots, p_{b_i -1}, p_{b_i}\}$$
and $\bigcup_{i=0}^m \mathcal{C}_i$ are interlaced iff there is only one equivalence class.

We now prove the proposition. Suppose there exists more than one equivalence class, i.e. we have
$$\mathcal{E}_1 = \{p_1, \dots, p_a\};\ \ \ \mathcal{E}_2 = \{p_{a+1}, \dots, p_b\}$$
for some $1 \leq a < n$. Since the smallest element in any equivalence class
must be the smallest element of a chain, and the largest element in a class must be
the largest element of a chain, we have
$$\mathcal{C}_i = \{\ \dots,\ p_{a}\}\ \ \{p_{a+1},\ \dots \} = \mathcal{C}_j.$$
By the above description of $s \in S_n$, it is obvious that $s \in S_a \times S_{n-a} \subset S_n$,
which does not involve the simple reflection $s_a$.

Conversely, if there is only one equivalence class, we suppose on the contrary that
there exists  some $1\leq a < n$ such that $s \in S_a \times S_{n-a}$. Since $p_a, p_{a+1}$
are in the same equivalence class, then at least one of the following
$$\{p_a, p_{a+1}\},\ \ \ \ \ \{p_a, p_{a+2}\}, \ \ \ \ \ \{p_{a-1}, p_{a+1}\}$$
is in the same chain $\mathcal{C}_i$ for some $0 \leq i \leq m$. By `flipping' $\mathcal{C}_i$ in either case, there must be some $u \leq a < a+1 \leq v$ such $s = \begin{pmatrix} \dots & u & \dots & v & \dots \\ \dots & v & \dots & u & \dots \end{pmatrix}$.
The reduced expression of such $s$ must involve the simple reflection $s_a$, hence we obtain a contradiction.
Therefore, $s$ must involve all simple reflections in its reduced expression.
\end{proof}

\begin{example} \label{eg:interlaced}
Consider the interlaced chain with smallest coordinate $1$
given in Example \ref{eg:spinlowest}:
\begin{align*}
\{10 && && 8\} && && \{6\} && && \{4\} \\
&& \{9 && && 7 && && 5 && && 3 && && 1\}
\end{align*}
Its corresponding irreducible
representation in $SL(9)$ has Langlands parameter $(\lambda',-s\lambda')$,
where $s = \begin{pmatrix}1 & 2 & 3 & 4 & 5 & 6 & 7 & 8 & 9 \\ 3 & 9 & 1 & 8 & 5 & 6 & 7 & 4 & 2 \end{pmatrix}$, and
$\lambda'$ $=$ $[1/2,1/2,1/2,1/2,1/2,1/2,1]$, where $[a_1,\dots,a_m]$ is defined by
$$[a_1,\dots,a_m] := a_1\varpi_1 + \dots, + a_m\varpi_m.$$
In fact, the coordinates of $\lambda'$ is simply obtained by taking the
difference of the neighboring coordinates of $\lambda = \frac{1}{2}(10,9,8,7,6,5,4,3,1)$.

The calculation in Example \ref{eg:spinlowest} implies that the
spin-lowest $K$-type for $J(\lambda',-s\lambda')$ in $SL(8)$ is
$V_{[1,1,1,2,0,1,1,1]}$.
\end{example}

\begin{example} \label{spherical}
We explore the possibilities of chains $\bigcup_{i=0}^m \mathcal{C}_i$
whose corresponding Zhelobenko parameter $(\lambda',-s\lambda')$ gives a spherical representation.

In order for the lowest $K$-type to be trivial,
we need the $\mathcal{T}_i$'s in Algorithm \ref{alg:spinlkt} to have the same
average value $k_i$ for all $i$, that is, the mid-point of all $\mathcal{C}_i$'s (if there
is more than one) must be the same. This leaves the possibility
of $\bigcup_{i=0}^m \mathcal{C}_i$ consisting of a single chain,
which corresponds to the trivial representation, or there
are two chains of lengths $a > b > 0$ whose entries are of
different parity. Hence it must be of the form
$$\{2a-1, 2a-3, \dots, 3, 1\} \cup \{a+(b-1), a+(b-3), \dots, a-(b-3), a-(b-1)\},$$
where $a$, $b$ are of different parity.

In other words, such representations can only occur for $SL(n)$
with $n = a+b$ is odd, and is equal to $Ind_{S(GL(a) \times GL(b))}^{SL(n)}(\mathrm{triv} \otimes \mathrm{triv})$,
which are the unipotent representations corresponding to the nilpotent orbit
with Jordan block $(2^b1^{a-b})$ (see \cite[Section 5.3]{BP}).
Its Langlands parameter $(\lambda',-s\lambda')$ has $2\lambda' = [\underbrace{2,\dots,2}_{(a-b-1)/2},\underbrace{1,\dots,1}_{2b},\underbrace{2,\dots,2}_{(a-b-1)/2}]$ and $s = w_0$ (see \cite[Conjecture 5.6]{DD}).
Moreover, its spin-lowest $K$-type is given by Equation (5.5) of \cite{BP}, which
matches with our calculations in Algorithm \ref{alg:spinlkt}.
\end{example}

For the rest of this section, we give two applications of Proposition \ref{prop:scattered}:
\subsection{The spin-lowest $K$-type is unitarily small}
To offer a unified conjectural description of the unitary dual, Salamanca-Riba and Vogan
formulated the notion of unitarily small (\emph{u-small} for short) $K$-type in \cite{SV}.
Here we only quote it for a complex connected simple Lie group $G$ -- using the setting in the introduction, a $K$-type $V_{\delta}$ is u-small if and only if $\langle \delta-2\rho, \varpi_i\rangle\leq 0$ for $1\leq i\leq \mathrm{rank}(\fg_0)$ (see Theorem 6.7 of \cite{SV}).

\begin{lemma}\label{lemma-u-small}
Let $\lambda=\sum_{i=1}^{\mathrm{rank}(\fg_0)}\lambda_i\varpi_i \in\fh_0^*$ be a dominant weight such that
$\lambda_i=\frac{1}{2}$ or $1$ for each $1\leq i\leq n$, and $V_{\delta}$ be the $K$-type with highest weight $\delta$ such that
$$
\{\delta-\rho\}=2\lambda-\rho.
$$
Then $\langle\delta-2\rho, \varpi_i\rangle\leq 0$, $1\leq i\leq \mathrm{rank}(\fg_0)$. Therefore, the $K$-type $V_{\delta}$ is u-small.
\end{lemma}
\begin{proof}
By assumption, there exists $w\in W$ such that $\delta=w^{-1}(2\lambda-\rho)+\rho$. Thus
\begin{align*}
\langle \delta-2\rho, \varpi_i\rangle &=\langle w^{-1}(2\lambda-\rho)-\rho, \varpi_i\rangle\\
&=\langle w^{-1}(2\lambda-\rho), \varpi_i\rangle-\langle\rho, \varpi_i\rangle\\
&=\langle 2\lambda-\rho, w(\varpi_i)\rangle -\langle\rho, \varpi_i\rangle.
\end{align*}
On the other hand, let $w=s_{\beta_1}s_{\beta_2}\cdots s_{\beta_p}$ be a reduced decomposition of $w$ into simple root reflections. Then by Lemma 5.5 of \cite{DH},
\begin{equation}
\varpi_i-w(\varpi_i)=\sum_{k=1}^p \langle \varpi_i, \beta_k^\vee \rangle s_{\beta_1}s_{\beta_2}\cdots s_{\beta_{k-1}}(\beta_k).
\end{equation}
Note that $s_{\beta_1}s_{\beta_2}\cdots s_{\beta_{k-1}}(\beta_k)$ is a positive root for each $k$. Now we have that
\begin{align*}
\langle \delta-2\rho, \varpi_i\rangle &=\Big\langle 2\lambda-\rho,\varpi_i- \sum_{k=1}^p \langle \varpi_i, \beta_k^\vee \rangle s_{\beta_1}s_{\beta_2}\cdots s_{\beta_{k-1}}(\beta_k)\Big\rangle -\langle\rho, \varpi_i\rangle\\
&=2\langle\lambda -\rho, \varpi_i\rangle-\sum_{k=1}^p \langle \varpi_i, \beta_k^\vee \rangle \langle 2\lambda-\rho,  s_{\beta_1}s_{\beta_2}\cdots s_{\beta_{k-1}}(\beta_k) \rangle \\
&\leq  2\langle \lambda-\rho, \varpi_i\rangle\\
&\leq 0.
\end{align*}
\end{proof}

\begin{corollary} \label{cor-u-small}
The unique spin-lowest $K$-type $V_{\tau}$ of any scattered representation of $SL(n)$ is u-small. Consequently,
Conjecture C of \cite{DD} holds for $SL(n)$.
\end{corollary}
\begin{proof}
Let $(\lambda, -s\lambda)$ be the Zhelobenko parameter for a scattered representation of $SL(n)$. Write $\lambda=\sum_{i=1}^{n-1} \lambda_i \varpi_i$ in terms of the fundamental weights. Then it is direct from our definition of the interlaced chains that each $\lambda_i$ is either $\frac{1}{2}$ or $1$ (recall Proposition \ref{prop:scattered} and Example \ref{eg:interlaced}). Let $V_{\tau}$ be the unique spin-lowest $K$-type of the scattered representation. Then $\{\tau-\rho\}=2\lambda-\rho$ (see Proposition \ref{prop-spin-lowest}). Thus the result follows from Lemma \ref{lemma-u-small}.
\end{proof}

\subsection{Number of scattered representations}
As another application of Proposition \ref{prop:scattered}, we compute the number of
scattered representations of $SL(n)$. By the proposition, it is equal to the number
of interlaced chains with $n$ entries with the smallest entry equal to $1$.
We now give an algorithm of constructing new interlaced chains with smallest
coordinate equal to $1$ from those with one less coordinate:

\begin{algorithm} \label{interlaced}
Let $\displaystyle \bigcup_{i=1}^p \{2A_i-1,\dots, 2a_i-1\} \cup \bigcup_{j=1}^q \{2B_j, \dots, 2b_j\}$ be a union of interlaced chains with
such that
\begin{itemize}
\item $A_{i'} > A_i$ if $i' > i$, and $B_{j'} > B_j$ if $j' > j$; and\\
\item $2a_p - 1 = 1$
\end{itemize}
We construct two new interlaced chains with one extra coordinate as follows. (When $q=0$, we adopt CASE I only.)\\

\noindent{\bf CASE I:} If $2A_p -1 > 2B_q + 1$, then the two new interlaced chains are
\begin{align*}
&& && && && \{2B_q&& && \dots && && 2b_q\} && && \dots \\
\{{\bf 2A_p+1} && 2A_p-1&& && \dots && && 2a_p-1\} && &&  \dots
\end{align*}
and
\begin{align*}
&& \{{\bf 2A_p-2}\} && && \{2B_q&& && \dots && && 2b_q\} && && \dots \\
\{2A_p-1&& \dots && \dots && && 2a_p-1\} && &&  \dots
\end{align*}

\noindent{\bf CASE II:} If $2A_p-1 = 2B_q+1$, then the two new interlaced chains are
\begin{align*}
&& && \{2B_q&& && \dots && && 2b_q\} && && \dots \\
\{{\bf 2A_p+1} && 2A_p-1&& && \dots && && 2a_p-1\} && &&  \dots
\end{align*}
and
\begin{align*}
\{{\bf 2B_q+2} && 2B_q&& && \dots && && 2b_q\} && && \dots \\
& \{2A_p-1 && && \dots && && 2a_p-1\} && &&  \dots
\end{align*}

\noindent{\bf CASE III:} If $2A_p-1 = 2B_q - 1$, then the two new interlaced chains are
\begin{align*}
& \{2B_q && && \dots && && 2b_q\} && &&  \dots \\
\{{\bf 2A_p+1} && 2A_p-1&& && \dots && && 2a_p-1\} && && \dots
\end{align*}
and
\begin{align*}
\{{\bf 2B_q + 2} && 2B_q&& && \dots && && 2b_q\} && &&  \dots \\
&& && \{2A_p-1&& && \dots && && 2a_p-1\} && && \dots
\end{align*}

\noindent{\bf CASE IV:} If $2A_p-1 < 2B_q - 1$, then the two new interlaced chains are
\begin{align*}
\{2B_q&& \dots && \dots && && 2B_q\} && &&  \dots \\
&& \{{\bf 2B_q-1}\} && && \{2A_p-1&& && \dots && && 2a_p-1\} && && \dots
\end{align*}
and
\begin{align*}
\{{\bf 2B_q+2} && 2B_q&& && \dots && && 2b_p\} && &&  \dots \\
&& && && && \{2A_p-1&& && \dots && && 2a_p-1\} && && \dots
\end{align*}

\end{algorithm}

\begin{example}
Suppose we begin with an interlaced chain $\{9,7,5,3,1\} \cup \{4,2\}$. Then the
new interlaced chains with one extra coordinate are
$$\{11,9,7,5,3,1\} \cup \{4,2\} \ \  \text{and}\ \ \{9,7,5,3,1\} \cup \{8\} \cup \{4,2\}.$$
\end{example}

\begin{proposition}
All interlaced chains with $n \geq 2$ entries with smallest coordinate equal to $1$
can be obtained uniquely from the chain $\{3\ 1\}$ by inductively applying
the above algorithm.
\end{proposition}
\begin{proof}
Suppose $\bigcup_{i=0}^m \mathcal{C}_i$ be interlaced chains with largest coordinate equal to $M \in \mathcal{C}_0$.
We remove a coordinate from it by the following rule: If $\mathcal{C}_i \neq \{M-1\}$ for all $i$, remove the
entry $M$ from $\mathcal{C}_0$. Otherwise, remove the whole chain $\{M-1\}$ from the original interlaced chains.

One can easily check from the definition of interlaced chain that the reduced chains
are still interlaced, and one can recover the original chain by applying Algorithm \ref{interlaced} on the reduced chain.

Therefore, for all interlaced chains with smallest entry $1$, we can use the reduction mentioned in the first paragraph
repeatedly to get an interlaced chain with only $2$ entries, which must be of the form $\{3\ 1\}$, and
repeated applications of Algorithm \ref{interlaced} on $\{3\ 1\}$ will retrieve the original interlaced chains
(along with other chains). In other words, all interlaced chains with smallest entry $1$ can be
obtained by Algorithm \ref{interlaced} inductively on $\{3\ 1\}$.

\smallskip

We are left to show that all interlaced chains are uniquely constructed using the algorithm -- Suppose on the contrary that
there are two different interlaced chains that give rise to the same $\bigcup_{i=0}^m \mathcal{C}_i$
after applying Algorithm \ref{interlaced}. By the algorithm, these two chains must be obtained from
$\bigcup_{i=0}^m \mathcal{C}_i$ by removing its largest odd entry $M_o \in \mathcal{C}_p$ or largest even entry $M_e \in \mathcal{C}_q$.
So they must be equal to
$$\bigcup_{i \neq p} \mathcal{C}_i  \cup (\mathcal{C}_p \backslash \{M_o\})\ \ \ \text{and}\ \ \ \bigcup_{i \neq q} \mathcal{C}_i  \cup (\mathcal{C}_q \backslash \{M_e\})$$
respectively.

Assume $M_o > M_e$ for now (and the proof for $M_e > M_o$ is similar). By applying
Algorithm \ref{interlaced} to $\bigcup_{i \neq q} \mathcal{C}_i  \cup (\mathcal{C}_q \backslash \{M_e\})$,
we obtain two interlaced chains
$$\bigcup_{i \neq p,q} \mathcal{C}_i  \cup \mathcal{C}_p' \cup (\mathcal{C}_q \backslash \{M_e\})\ \ \ \text{and}\ \ \ \bigcup_{i \neq q} \mathcal{C}_i  \cup (\mathcal{C}_q \backslash \{M_e\}) \cup \{M_o -1\},$$
where $\mathcal{C}_p' := \{M_o +2, \overbrace{M_o, \dots, m_o}^{\mathcal{C}_p}\}$.
Note that none of the above gives rise to the interlaced chains $\bigcup_{i=0}^m \mathcal{C}_i$:
Even in the case when $M_0 -1 = M_e$, $(\mathcal{C}_q \backslash \{M_e\}) \cup \{M_o -1\}$ and
$\mathcal{C}_q$ are different -- although they have the same coordinates, the first consists of two
chains while the second consists of one chain only. So we have a contradiction, and the result follows.
\end{proof}

\begin{corollary}\label{cor-number}
The number of interlaced chains with $n$ coordinates
and the smallest coordinate equal to $1$ is equal to $2^{n-2}$.
\end{corollary}

Since the scattered representations of $SL(n+1)$ are in one to one correspondence
with interlaced chains with $n+1$ coordinates having smallest coordinate $1$, this corollary implies that the number of scattered representations of Type $A_n$ is equal to
$2^{n-1}$. This verifies Conjecture 5.2 of \cite{D2}. Moreover, by using \texttt{atlas},
the spin-lowest $K$-types for all scattered representations
of $SL(n)$ with $n \leq 6$ are given in Tables 1--3 of \cite{D2}. One can
easily check the results there match with our $V_{\tau}$ in Algorithm \ref{alg:spinlkt}.

\begin{example}\label{exam-scattered-small-rank}
Let us start from $SL(2, \bC)$ and the chain $\{3\quad 1\}$. This chain corresponds to the trivial representation.

Now we consider $SL(3, \bC)$. By Algorithm  \ref{interlaced}, the chain $\{3\quad 1\}$ for $SL(2)$   produces two chains

$$\{5\quad 3 \quad 1\} \qquad\qquad\qquad\qquad \begin{aligned}
 \{ &2 \} \\
 \{ 3 \ \ \ &\ \ \ \ 1\}
\end{aligned}.$$

The first corresponds to the trivial representation, while the second gives the representation with $\lambda=[\frac{1}{2}, \frac{1}{2}]$ and $s = \begin{pmatrix}1 & 2 & 3 \\ 3 & 2 & 1 \end{pmatrix}$. One computes by Algorithm \ref{alg:spinlkt} that the spin-lowest $K$-type $\tau=[1, 1]$.

\smallskip

Now let us consider $SL(4)$. By Algorithm  \ref{interlaced}, the chain $\{5\quad 3\quad 1\}$ for $SL(3)$ produces two chains
$$\{7\quad 5\quad 3 \quad 1\} \ \qquad\qquad\qquad\qquad
\begin{aligned}
 \{ &4 \}  \\
\{5 \ \ \ &\ \ \ \ 3 \ \ \ \ \ \ \ 1\}
\end{aligned}.$$
The first chain corresponds to the trivial representation, while the second one gives the representation with  $\lambda=[\frac{1}{2}, \frac{1}{2}, 1]$ and $s = \begin{pmatrix}1 & 2 & 3 & 4 \\ 4 & 2 & 3 & 1 \end{pmatrix}$. One computes by Algorithm \ref{alg:spinlkt} that the spin-lowest $K$-type $\tau=[2, 0, 1]$. The other chain of $SL(3)$ shall produce
\begin{align*}
  \{ &2 \}  &\ \qquad \{4 \ \ \ \ &\ \ \ 2\}\\
\{5 \ \ \ \ \ \ \ 3 \ \ \ &\ \ \ \ 1\}      &\ \qquad         \{&3 \ \ \ \ \ \ \ 1\}
\end{align*}
One computes that $\lambda=[1, \frac{1}{2}, \frac{1}{2}]$, $s = \begin{pmatrix}1 & 2 & 3 & 4 \\ 4 & 2 & 3 & 1 \end{pmatrix}$,  $\tau=[1, 0, 2]$; and that $\lambda=[\frac{1}{2}, \frac{1}{2}, \frac{1}{2}]$, $s = \begin{pmatrix}1 & 2 & 3 & 4 \\ 3 & 4 & 1 & 2 \end{pmatrix}$,  $\tau=[1, 1, 1]$, respectively. These four representations (and their spin-lowest $K$-types) match precisely with Table 1 of \cite{D2}.
\end{example}

\medskip

\centerline{\scshape Acknowledgements}
We thank the referee sincerely for very careful reading and nice suggestions.

\medskip

\centerline{\scshape Funding}
Dong was supported by the National Natural Science Foundation of China (grant 11571097, 2016--2019). Wong is supported by the National Natural Science Foundation of China (grant 11901491) and the Presidential Fund of CUHK(SZ).


\begin{thebibliography}{PRV}

\bibitem[BDW]{BDW} D.~Barbasch, C.-P.~Dong, K.D.~Wong,
{\em A multiplicity one theorem for spin lowest $K$-types},
in preparation.

\bibitem[BP]{BP}
D.~Barbasch, P. Pand\v zi\'c,
{\em Dirac cohomology and unipotent representations of complex groups},
Noncommutative geometry and global analysis,
Contemp. Math. \textbf{546}, Amer. Math. Soc., Providence, RI, 2011, pp.~1--22.



\bibitem[DD]{DD} J.~Ding, C.-P.~Dong,
\emph{Unitary representations with Dirac cohomology: a finiteness result for complex Lie groups},
Forum Math. \textbf{32} (4) (2020), 941--964.

\bibitem[D1]{D1}  C.-P.~Dong,
\emph{On the Dirac cohomology of complex Lie group representations},
Transform. Groups \textbf{18} (1) (2013), 61--79. Erratum: Transform.
Groups \textbf{18} (2) (2013), 595--597.

\bibitem[D2]{D2}   C.-P.~Dong, \emph{Unitary representations with non-zero Dirac cohomology for complex $E_6$}, Forum Math. {\bf 31} (1) (2019), 69--82.

\bibitem[DH]{DH} C.-P.~Dong, J.-S.~Huang,  \emph{Jacquet modules and Dirac cohomology},
 Adv.  Math. \textbf{226} (4) (2011), 2911--2934.

\bibitem[GW]{GW}   R.~Goodman, N.~Wallach, \emph{Symmetry, representations, and invariants}, Graduate Texts in Mathematics, vol.~\textbf{255}. Springer, Dordrecht, 2009.

\bibitem [HKP]{HKP} J.-S.~Huang, Y.-F.~Kang, P.~Pand\v zi\'c, \emph{Dirac
cohomology of some Harish-Chandra modules}, Transform. Groups.
\textbf{14} (2009), no.1, 163--173.


\bibitem[HP1]{HP1}
J.-S. Huang, P. Pand\v{z}i\'{c},
{\it Dirac cohomology, unitary representations and a proof of a conjecture of
Vogan}, J. Amer. Math. Soc. \textbf{15} (2002), 185--202.

\bibitem[HP2]{HP2}
J.-S. Huang, P. Pand\v{z}i\'{c},
{\it Dirac Operators in Representation Theory},
Mathematics: Theory and Applications, Birkhauser, 2006.



\bibitem[HPP]{HPP} J.-S.~Huang, P.~Pand\v zi\'c, V.~Protsak, \emph{Dirac
cohomology of Wallach representations}, Pacific J. Math.   \textbf{250} (1) (2011), 163--190.

\bibitem[P1]{P1} R.~Parthasarathy, \emph{Dirac operators and the discrete
series}, Ann. of Math. \textbf{96} (1972), 1--30.

\bibitem[P2]{P2} R.~Parthasarathy, \emph{Criteria for the unitarizability of some highest weight modules},
Proc. Indian Acad. Sci. \textbf{89} (1) (1980), 1--24.

\bibitem[PRV]{PRV} R.~Parthasarathy, R.~Ranga Rao,
S.~Varadarajan, \emph{Representations of complex semi-simple Lie
groups and Lie algebras}, Ann. of Math. \textbf{85} (1967),
383--429.

\bibitem[SV]{SV} S.~Salamanca-Riba, D.~Vogan, \emph{On the classification of unitary representations of reductive Lie
groups}, Ann. of Math. \textbf{148} (3) (1998), 1067--1133.

\bibitem[V1]{V1}
D.~Vogan,
{\em The unitary dual of $GL(n)$ over an archimedean field},
Invent. Math. \textbf{83} (1986), 449--505.

\bibitem[V2]{V2}
D.~Vogan,
{\it Dirac operator and unitary representations},
3 talks at MIT Lie groups seminar, Fall 1997.



\bibitem[Zh]{Zh} D.~P.~Zhelobenko, \emph{Harmonic analysis on complex semisimple Lie
groups}, Mir, Moscow, 1974.
\end{thebibliography}
\end{document}